\documentclass[a4paper,11pt]{amsproc}
\usepackage{amsmath, amssymb, amsthm} 

\usepackage[colorlinks, backref]{hyperref}
\usepackage{amsrefs}
\usepackage{todonotes, nicefrac}
\usepackage{graphicx} 
\usepackage{tikz,tikz-cd}

\theoremstyle{plain}
\newtheorem{theorem}{Theorem}[section]
\newtheorem{question}[theorem]{Question}
\newtheorem{lemma}[theorem]{Lemma}
\newtheorem{corollary}[theorem]{Corollary}
\newtheorem{proposition}[theorem]{Proposition}

\theoremstyle{definition}
\newtheorem{definition}[theorem]{Definition}

\theoremstyle{remark}
\newtheorem{remark}[theorem]{Remark}

\DeclareMathOperator{\Sym}{Sym}
\newcommand{\cP}{{\mathcal P}}
\newcommand{\bbN}{{\mathbb N}}

\DeclareMathOperator{\ad}{ad}
\DeclareMathOperator{\Aut}{Aut}
\DeclareMathOperator{\Fin}{Fin}
\newcommand{\cU}{\mathcal U}

\DeclareMathOperator{\Th}{Th}
\DeclareMathOperator{\Cli}{Cl_{2}}
\newcommand{\CH}{{\mathsf {CH}}}
\newcommand{\ZFC}{{\mathsf {ZFC}}}
\newcommand{\OCA}{{\mathsf {OCA}}}
\newcommand{\PFA}{{\mathsf {PFA}}}
\newcommand{\MAsl}{{\mathsf {MA}_{\aleph_1}(\sigma\textrm{-linked})}}
\newcommand{\MA}{\mathsf {MA}_{\aleph_1} }

\title[On automorphism groups of metric reduced products]{On automorphism groups of metric reduced products of symmetric groups}

\author{Ben De Bondt}
\address{Ben De Bondt, Universität Münster\\
	Institut für Mathematische Logik und Grundlagenforschung\\ 
	Einsteinstr.\ 62\\
	48149 Münster,
	Germany}
\email{bdebondt@uni-muenster.de}

\author{Andreas Thom}
\address{Andreas Thom, Fakult\"at f\"ur Mathematik, TU Dresden, 01062 Dresden} 
\email{andreas.thom@tu-dresden.de}

\begin{document}
	
	\begin{abstract}
		We study isomorphisms between metric reduced products of symmetric groups with the normalized Hamming metric assuming the open coloring axiom $\OCA$ and Martin's axiom for $\sigma$-linked posets.
	\end{abstract}
	
	\maketitle
	
	\section{Introduction}
	
	The metric first order theory in the sense of continuous logic \cite{benyaacov} of symmetric groups ${\rm Sym}(n)$ with the normalized Hamming metric
	$$d_n(\sigma,\tau)= \frac{1}{n} |\{1 \leq i \leq n 
	\mid \sigma(i) \neq \tau(i)\}|$$
	is (even if not always under this name) of considerable importance in various fields of mathematics, notably in the study of approximation properties of infinite groups, i.e., Gromov's notorious question \cite{MR1694588} whether all groups are sofic, in ergodic theory of group actions, combinatorics and graph theory. 
	
	When $(k_n)_n$ is any sequence of natural numbers, we denote by $\Sym[(k_n)_n]$ the metric reduced product of the sequence of metric groups $(\Sym(k_n): n \in \bbN)$, see Section \ref{sec:def} for a definition. We assume throughout that $(k_n)_n$ is a sequence of natural numbers with $\lim_n k_n = \infty,$ unless explicitly stated otherwise. A basic result in the study of sofic groups \cites{MR2460675, MR2178069, MR3966829} says that a countable group is sofic if and only if it embeds into a metric ultraproduct or a metric reduced product of symmetric groups as above. Therefore, those groups are sometimes called \emph{universal sofic groups}. Due to these facts and simply in and of themselves, it is worthwhile to study these objects and the isomorphisms between them.
	It has been shown in \cite{alekseevthom} that the metric theories of the groups ${\rm Sym}(n)$ do not converge as $n$ tends to infinity and in fact, there are infinitely many metric FO-sentences that can be assigned arbitrary values in $\{0,1\}$. In particular, $2^{\aleph_0}$ is a lower bound both for the number of isomorphism classes of metric ultraproducts and the number of isomorphism classes of metric reduced products of such groups (see Corollary \ref{cor.noniso}); a bound which in the latter case clearly is attained. It has long been observed however that behavior of isomorphism between the metric reduced product groups $\Sym[(k_n)_n]$ depends heavily on the set theory beyond ${\mathsf {ZFC}}$. In presence of the continuum hypothesis ${\mathsf {CH}}$, for example, isomorphism simply boils down to elementary equivalence of the metric theories and every group $\Sym[(k_n)_n]$ has as many automorphisms as possible (see Theorem~\ref{Th.CHcase}). 
	
	In this note, we want to clarify the situation under a different natural set of additional axioms, namely Todor\v{c}evi\'c's open coloring axiom $\OCA$ and the fragment of Martin's axiom $\MAsl$.
	Both these axioms appear in set theory within the realm of so called forcing axioms, being consequences of the (stronger) proper forcing axiom $\PFA$. 
	Adding moderately weak forcing axioms as considered in this note produces equiconsistent extensions of $\ZFC$, which already provide a coherent picture of a rich set theoretic universe in which $\CH$ fails. 
	This makes them natural and valuable tools for studying, or proving consistency of, various particular principles that oppose $\CH$. In a main line of applications, these particular principles express rigidity behavior of various massive quotient structures that are typically failing under $\CH$. Examples include quotient Boolean algebras, several other algebraic reduced products of countable first order structures, metric quotient structures including $\mathrm{C^*}$-algebras (notably corona algebras and the Calkin algebra), and in the case at hand, universal sofic groups of the reduced metric product type. See \cite{farah2024coronarigidity} for an overview and detailed discussion on the impact of these axioms on rigidity of such structures.
	Referring to \cite{kunen2011set} for set-theoretic background on, as well as the precise statements of, these two particular principles, we note that $\MA$ (of which $\MAsl$ is a fragment) is equivalent to the statement that the intersection of $\aleph_1$-many open dense subsets of a ccc compact topological Hausdorff space is non-empty, whereas $\OCA$ is a Ramsey-like statement governing the behavior of open colorings of (the complete undirected graph on) separable metric spaces. 
	
	The conjunction of $\OCA$ and $\MAsl$ has recently been shown in \cite{TrivIso} and \cite{metricliftingtheorem}, to imply that 
	so called \emph{coordinate-respecting} properties of maps between reduced products $\prod_n M_n / \Fin$ and $\prod_n N_n / \Fin$ (of either algebraic or metric type) of separable structures $M_n, N_n$ imply that the maps are necessarily induced (possibly after permutation of coordinates) by coordinatewise application of a sequence $(h_n)_n$ of functions $h_n: M_n \to N_n.$ The main result of \cite{metricliftingtheorem} in particular provides the main set theoretic ingredient that we are making use of.
	
	Our main result expresses that assuming the two axioms $\OCA$ and $\MAsl$
	all isomorphisms between groups $\Sym[(k_n)_n]$ and $\Sym[(l_n)_n]$ are trivial, where we need to be careful to define trivial in the correct way. 
	Here, triviality of isomorphisms encompasses inner automorphisms that arise from conjugation in each coordinate as well as isomorphisms that arise from the observation that $\Sym(n)$ and $\Sym(n + 1)$ are almost isometrically isomorphic as metric groups for $n$ large. See Definition~\ref{def.triv} for a precise formulation.
	We also give an explicit description of the outer automorphism group of a metric reduced product of symmetric groups.
	
	\section{The basic setup}
	\subsection{Metric reduced products}
	\label{sec:def}
	Let $(k_n)_n$ be a sequence of natural numbers. We define the metric reduced product of ${\rm Sym}(k_n)$ to be the product $\prod_{n}{\rm Sym}(k_n)$ modulo the normal subgroup
	$N((k_n)_n):= \{ (a_n)_n \mid \limsup_n d_n(a_n,1_n) = 0 \}.$
	We denote this group by ${\rm Sym}[(k_n)_n]$.
	
	In our notation we usually won't distinguish between the element $(a_n)_n \in \prod_n \Sym(k_n)$ and the element $[(a_n)_n]$ of $\Sym[(k_n)_n]$ that it represents. This is an affordable convenience as the context will each time make clear the intended type of the objects under consideration. 
	An \emph{almost permutation} of $\bbN$ is a map $f: \bbN \to \bbN$ which has cofinite range and has the property that there exists a cofinite set $S \subseteq \bbN$ such that $f \upharpoonright S$ is injective. Note that the almost permutations of $\bbN$, considered modulo equality on cofinite sets, form a group under composition.
	
	We denote by $\sigma \oplus \tau \in {\rm Sym}(n+m)$ the disjoint union of two permutations $\sigma \in {\rm Sym}(n), \tau \in {\rm Sym}(m)$.
	
	\subsection{Trivial automorphisms of $\Sym[(k_n)_n]$}
	The inner automorphisms provide the simplest examples of automorphisms of  $\Sym[(k_n)_n]$.
	Slightly less obvious is that, at least for certain sequences $(k_n)_n,$ the group $\Sym[(k_n)_n]$ can also have automorphisms that do move its conjugacy classes. It is natural to attempt to produce such automorphisms by starting from a suitable almost permutation $f$ of $\bbN$, to permute the terms of any element of $\Sym[(k_n)_n]$. Thus, one is tempted to define an automorphism $\psi_f$ that maps $(a_n)_n \in \Sym[(k_n)_n]$ to $(a_{f(n)})_n$.  Of course, for this, one should make sure to make sense of the latter object $(a_{f(n)})_n$ as a well-defined element of $\Sym[(k_n)_n]$. This is not immediate, and in fact will only be possible if $f$ has suitable asymptotic behavior with respect to the sequence $(k_n)_n$. Indeed: for every $n$ the permutation $a_{f(n)} \in \Sym(k_{f(n)})$ could have to be lifted or cut down to an element of $\Sym(k_n)$ and in addition the resulting element, after all such cutting and lifting, of $\Sym[(k_n)_n]$ should not depend on the choice of the representative of the element $(a_n)_n \in \Sym[(k_n)_n]$. To handle this, we consider the following operations.
	
	\begin{definition} Let $l \leq m \leq n$ be natural numbers and $\sigma \in \Sym(m)$ a permutation.
		\begin{itemize}
			\item[(cutting)]     Define $\sigma \downarrow \Sym(l)$ to be the permutation of $\{1, \ldots, l\}$ that maps each $i \in \{1, \ldots, l\}$ to $\sigma^{k_i}(i),$ where $k_i \geq 1$ is minimal such that $\sigma^{k_i}(i) \in \{1, \ldots, l\}$.
			\item[(lifting)] Define $\sigma \uparrow \Sym(n)$ to be the permutation $\sigma \oplus 1_{n-m}$ of $\{1, \ldots, n\}$ which acts on $\{1, \ldots, m\}$ as $\sigma$ does and fixes all elements of $\{m+1, \ldots, n\}.$
			
		\end{itemize}
		
	\end{definition}
	
	\begin{definition}
		For any $m,n \in \bbN$ and permutation $\sigma \in \Sym(m)$, define $\sigma \updownarrow \Sym(n)$ as follows:
		\[\sigma \updownarrow\Sym(n)=
		\begin{cases}\sigma\uparrow\Sym(n)&\text{if } m\leq n,\\
			\sigma\downarrow \Sym(n)&\text{if } n\leq m.
		\end{cases}
		\]
	\end{definition}
	
	\begin{lemma}
		Let $ m \leq n$ be natural numbers. 
		\begin{enumerate}
			\item The map $\Sym(n) \to \Sym(m): \sigma \to \sigma \downarrow \Sym(m)$ is a uniform $\frac{2(n-m)}{m}$-homomorphism,
			\item for all $\tau \in \Sym(n)$ and $\sigma \in \Sym(n)$: $d_n(\sigma \downarrow \Sym(m) \uparrow \Sym(n), \sigma) \leq \frac{2(n-m)}{n}$  and $\tau \uparrow \Sym(n) \downarrow \Sym(m)= \tau.$
		\end{enumerate}
	\end{lemma}
	
	It is now clear that to every almost permutation $f$ of $\bbN$ which satisfies $\lim_n \frac{k_{f(n)}}{l_n}=1$, one can associate a (well-defined) isomorphism $\psi_f : \Sym[(k_n)_n] \to \Sym[(l_n)_n]$: $\psi_f$ maps $(a_n)_n \in \Sym[(k_n)_n]$ to the element $( a_{f(n)} \updownarrow \Sym(l_n) )_n$ of $\Sym[(l_n)_n].$
	
	We have thus obtained a natural family of isomorphisms between groups $\Sym[(k_n)_n]$ and $\Sym[(l_n)_n]$ that seem to cover all simple instances of constructing such isomorphisms.
	
	\begin{definition}\label{def.triv}
		Call an isomorphism $\varphi:  \Sym[(k_n)_n] \to\Sym[(l_n)_n]$ \emph {trivial} if it is, up to conjugation with an element of $\Sym[(l_n)_n]$, of the form $\psi_f:  \Sym[(k_n)_n] \to\Sym[(l_n)_n]$ for some almost permutation $f$ of $\bbN$ which satisfies $\lim_n \frac{k_{f(n)}}{l_n}=1$.
	\end{definition}
	
	The most basic example that illustrates the previous definition is the map $$\varphi \colon {\rm Sym}[(n)_n] \to {\rm Sym}[(n)_n],$$ which is defined as $\varphi((a_{n})_n) := (a_{n-1} \oplus 1_1)_n.$ This is clearly an isometric automorphism of the metric reduced product and evidently not inner. According to our definition it is trivial. Our main result basically says that, assuming $\OCA$ and $\MAsl$, automorphisms like this one are  the only ones we need to add to the inner automorphisms in order to obtain the full automorphism group.
	
	\section{Proof of the main result}
	
	\subsection{Isomorphisms of product form}
	We first analyze all isomorphisms $\varphi:\Sym[(k_n)_n] \to\Sym[(l_n)_n]$ which are of the following specific form. In this subsection, as well as the next, no set theoretic assumptions beyond those in $\mathsf{ZFC}$ are needed.
	
	\begin{definition}
		A map $\varphi:\Sym[(k_n)_n] \to\Sym[(l_n)_n]$ is said to be \emph{of product form} if there exists a sequence $(h_n)_n$ of maps $h_n : \Sym(k_n) \to \Sym(l_n)$ such that $\varphi$ maps every element $(a_n)_n \in \Sym[(k_n)_n]$ to the element $(h_n(a_n))_n$ of $\Sym[(l_n)_n].$ 
	\end{definition}
	
	Clearly every inner automorphism of $\Sym[(k_n)_n]$ is of product form, we prove in this subsection that conversely every automorphism of $\Sym[(k_n)_n]$ which is of product form is inner. The argument depends on the following Ulam stability result for almost surjective almost homomorphisms between finite symmetric groups, which we derive from \cite{MR4634678}*{Theorem 1.2}.
	
	\begin{theorem} \label{Th.stab}
		There exists a universal constant $c>2$, such that for all $n,m \in \mathbb N$ and all $\varphi \colon {\rm Sym}(n) \to {\rm Sym}(m)$ the following holds for $\delta<1/(3c)$ and $n,m \geq 7$:
		If 
		\begin{enumerate}
			\item $d_m(\varphi(\sigma\tau),\varphi(\sigma)\varphi(\tau)) \leq \delta$ for all $\sigma,\tau \in {\rm Sym}(n)$ and
			\item for all $\sigma \in {\rm Sym}(m)$, there exists $\tau \in {\rm Sym}(n)$, such that $d_m(\varphi(\tau),\sigma)\leq\delta$,
		\end{enumerate}
		then we obtain
		$(1-\delta)m \leq n \leq (1+c \delta)m$ and
		there exists a permutation $\alpha \in {\rm Sym}(n)$, such that
		$$d_{n \vee m}(\varphi(\sigma) \uparrow \Sym(n\vee m), {\rm ad}_{\alpha}(\sigma) \uparrow \Sym(n \vee m) ) \leq 4c \delta, \quad \forall \sigma \in {\rm Sym}(n).$$
		Moreover, the map $\varphi$ is almost isometric, i.e.,
		$$\left|d_n(\sigma,\tau) - d_m(\varphi(\sigma),\varphi(\tau)) \right| \leq 5c\delta, \quad \forall \sigma,\tau \in {\rm Sym}(n).$$
	\end{theorem}
	\begin{proof}
		By (1) and \cite{MR4634678}*{Theorem 1.2} there exists a universal constant $c>0$, such that there exists a homomorphism $\psi \colon {\rm Sym}(n) \to {\rm Sym}(m')$ with $m' \in [m,(1+c \delta)m] \cap \mathbb N$ and $d_{m'}(\varphi(\sigma) \oplus 1_{m'-m},\psi(\sigma)) \leq c \delta$ for all $\sigma \in {\rm Sym}(n)$.
		
		The size of a $\delta$-neighborhood of the permutation $\varphi(\tau) \in {\rm Sym}(m)$ in the normalized Hamming metric is at most of cardinality
		$(\delta m)! \binom{m}{\delta m}.$ We conclude from the second assumption that
		$$m! \leq n! (\delta m)! \binom{m}{\delta m} = \frac{ n!m!}{((1-\delta)m)!}$$ and hence
		$(1-\delta)m \leq n.$
		
		Now, if $\psi$ factorizes over the group $\mathbb Z/2\mathbb Z$, the image of $\varphi$ is contained in a $c\delta$-neighborhood of at most two points and hence by our second assumption all of ${\rm Sym}(m)$ is contained in the $(c+1)\delta$-neighborhood of those two points. However, this implies
		$$m! \leq 2((c+1)\delta m)! \binom{m}{(c+1)\delta m} \leq  \frac{2m!}{((1-(c+1)\delta)m)!},$$
		which is absurd, since $(1-(c+1)\delta)m \geq 3$ by our choices.
		
		Thus we conclude that $\psi$ is injective. Then, we must have $n \leq m'$ and we obtain that
		$$(1-\delta)m \leq n \leq (1+c \delta)m.$$
		Moreover, the only permutation representations of ${\rm Sym}(n)$ for $n \geq 7$ on a set $\{1,\dots,m'\}$ with $m' \leq (1+c\delta)m \leq \frac{1 +c \delta}{1 - \delta} n < 2n$ are induced by embeddings $\alpha' \colon \{1,\dots,n\} \to \{1,\dots,m'\}$, where ${\rm Sym}(n)$ may act on the complement of the image of $\alpha'$ by an involution.
		Since the image of $\alpha'$ necessarily has a large overlap with the subset $\{1,\dots,n\}$, we may find $\alpha \in {\rm Sym}(n)$ 
		with $$d_{m'}(\psi(\sigma), {\rm ad}_{\alpha}(\sigma) \oplus 1_{m'-n}) \leq 3c\delta,\quad \forall \sigma \in {\rm Sym}(n).$$ Finally, this implies
		$$d_{n \vee m}(\varphi(\sigma) \oplus 1_{n\vee m - m}, {\rm ad}_{\alpha}(\sigma) \oplus 1_{n \vee m-n}) \leq 4c \delta, \quad \forall \sigma \in {\rm Sym}(n)$$
		as required.
	\end{proof}
	
	The following is a straightforward observation.
	
	\begin{lemma} \label{lem.defect}
		Let $((G_n,d_n) : n \in \bbN)$ and $((H_n,\partial_n) : n \in \bbN)$ be two sequences of bi-invariant metric groups of uniformly bounded diameter and $(h_n : n \in \bbN)$ a sequence of maps $h_n:G_n \to H_n$. Suppose that the map $\prod_n G_n \to \prod_n H_n$ which maps $(a_n)_n$ to $(h_n(a_n))_n$ induces a map $\varphi: \prod_n G_n / \Fin \to \prod_n H_n / \Fin .$ Then
		\begin{enumerate}
			\item $\varphi$ is a group homomorphism if and only if
			\[\lim_n \sup_{x,y\in G_n} \partial_n(h_n(x) h_n(y) , h_n(xy)) = 0,\]
			\item $\varphi$ is surjective if and only if
			\[\lim_n \sup_{x\in H_n} \inf_{y\in G_n} \partial_n(x , h_n(y)) =0.\]
		\end{enumerate}
	\end{lemma}
	
	Combining Theorem~\ref{Th.stab} and Lemma~\ref{lem.defect}, we obtain the following structure result for isomorphisms of product form.
	
	\begin{lemma}\label{lem.coordfixcase}
		Let $(k_n)_n$ and $(l_n)_n$ be two sequences of natural numbers with $\lim_n k_n = \lim_n l_n = \infty.$ If $\varphi : \Sym[(k_n)_n] \to \Sym[(l_n)_n]$ is a surjective group morphism of product form, then $\lim \frac{k_n}{l_n}=1$ and there exists $\sigma \in \Sym[(l_n)_n]$ such that $\varphi$ maps every element $(a_n)_n \in \Sym[(k_n)_n]$ to the element $\ad_\sigma((a_n \updownarrow \Sym(l_n))_n)$ of $\Sym[(l_n)_n].$ 
		In particular, $\varphi$ is automatically injective.
	\end{lemma}	
	
	\begin{proof}
		From Lemma~\ref{lem.defect}, it follows that for every $\delta>0,$ both assumptions of Theorem~\ref{Th.stab} are satisfied by all but finitely many of the maps $h_n.$ It follows that $\lim \frac{k_n}{l_n}=1$. Coordinatewise application of Theorem~\ref{Th.stab} yields an element $\sigma' \in \Sym[(k_n \vee l_n)_n]$ such that whenever $\varphi$ maps $(a_n)_n \in \Sym[(k_n)_n]$ to $(b_n)_n \in \Sym[(l_n)_n]$, we have  $(b_n \uparrow  \Sym(k_n \vee l_n))_n = \ad_{\sigma'}( (a_n\uparrow \Sym(k_n \vee l_n))_n)$. Because  $\lim \frac{k_n}{l_n}=1$, the element $\sigma = (\sigma_n)_n = (\sigma'_n \downarrow \Sym(l_n))_n$ of $\Sym[(l_n)_n]$ then has the required properties.
	\end{proof}
	
	\subsection{Permuting conjugacy classes}	
	
	Following notation from \cite{metricliftingtheorem}, every infinite $S\in \cP(\bbN) / \Fin$ induces a pseudometric $d^S$ on every metric reduced product $\Sym[(k_n)_n],$ defined as follows:
	\[
	d^S(a,b)=\limsup_{n\in S} d_n(a_n,b_n),
	\]
	for all $a=(a_n)_n,b=(b_n)_n \in \Sym[(k_n)_n]$.
	For $a,b \in \Sym[(k_n)_n]$, we write $a=_Sb$ for $d^S(a,b)=0$. Recalling the notion of triviality of isomorphisms from Definition~\ref{def.triv}, we find that every trivial isomorphism $\varphi: \Sym[(k_n)_n] \to \Sym[(l_n)_n]$ respects the pseudometrics $d^S$ in the following sense:
	\[ d^S(a,b) = d^{f^{-1}[S]}(\varphi(a),\varphi(b)),\]
	where $f$ is the corresponding almost permutation of $\bbN.$
	
	It turns out that (see Corollary \ref{cor.isometric}) a very similar condition is satisfied by every  $\varphi: \Sym[(k_n)_n] \to \Sym[(l_n)_n]$, provably so in $\mathsf{ZFC}$. This is remarkable as not every such isomorphism is necessarily trivial in $\mathsf{ZFC}$ (see further).
	
	\begin{theorem}\label{th.action boundary}	
		Every isomorphism $\varphi: \Sym[(k_n)_n] \to \Sym[(l_n)_n]$ induces an automorphism $\theta$ of the Stone-Cech boundary $\partial \beta \bbN,$ which moreover has the following two properties: 
		\begin{enumerate}
			\item\label{itm1} for all $a,b \in \Sym[(k_n)_n]$ and any infinite subset $S \subseteq \bbN,$ we have $a =_S b \Leftrightarrow \varphi(a) =_{\theta(S)} \varphi(b),$
			\item\label{itm2} for every nonprincipal ultrafilter $\cU \in \partial \beta \bbN$, the isomorphism $\varphi$ induces an isomorphism $\varphi_\cU$ between the metric ultraproduct groups $\prod_n  \Sym(k_n) / \cU$ and $\prod_n  \Sym(l_n) / \theta^{-1}(\cU).$
		\end{enumerate}
	\end{theorem}
	
	\begin{proof}
		Let $\Cli(\Sym[(k_n)_n])$ denote the set of conjugacy classes of involutions in the group $\Sym[(k_n)_n].$
		Consider the map $D_1$ which maps $(a_n)_n \in \Sym[(k_n)_n]$ to the element $D_1( (a_n)_n ) = (d_{k_n}(1,a_n))_n$ of the  metric reduced product $[0,1]^\bbN / \Fin$ (i.e.\ the set of $[0,1]$-valued sequences modulo convergence to zero).
		This induces a bijection, still denoted $D_1$, between the sets $\Cli(\Sym[(k_n)_n])$ and $[0,1]^\bbN / \Fin$. Expand $[0,1]^\bbN / \Fin$ to the structure \[ ([0,1]^\bbN / \Fin,0,1,\leq,\neg,\oplus),\] defined as follows: 
		\begin{itemize}
			\item $0$ is represented by the constant sequence $(0)_n$,
			\item $1$ is represented by the constant sequence $(1)_n$,
			\item $(x_n)_n \leq (y_n)_n$ if and only if for every $\varepsilon>0,$ $x_n < y_n+\varepsilon$ for cofinitely many $n\in \bbN,$
			\item  $\neg (x_n)_n =(1-x_n)_n,$
			\item $(x_n)_n \oplus (y_n)_n = ( \min(1,x_n+y_n))_n.$
		\end{itemize}
		
		We claim that this structure is, after translation to $\Cli(\Sym[(k_n)_n])$ via the map $D_1$, definable from the group structure on $\Sym[(k_n)_n]$. Firstly, one gets that for all $C_1,C_2 \in \Cli(\Sym[(k_n)_n]):$
		\[ D_1(C_1) \leq D_1(C_2) \text{ if and only if } (C_1)^2 \subseteq (C_2)^2.\] 
		
		Thus the entire lattice structure is definable over $\Sym[(k_n)_n]$. In addition,  it is straightforward to check that whenever $C_1,C_2 \in \Cli(\Sym[(k_n)_n])$, the element $D_1(C_1) \oplus D_1(C_2)$ of $[0,1]^\bbN / \Fin$ is precisely the maximum in $([0,1]^\bbN / \Fin, \leq)$ of the following set:
		\[ \{ D_1(C') : C' \in  \Cli(\Sym[(k_n)_n]) \text{ and }(\exists a \in C_1)(\exists b\in C_2)\, ab \in C' \}.\]
		Finally, for an involution $a \in \Sym[(k_n)_n],$ the element $\neg D_1(a)$ of $[0,1]^\bbN / \Fin$ is uniquely determined by being equal to $D_1(b)$ for some involution $b \in \Sym[(k_n)_n]$ that commutes with $a$, and satisfies $D_1(a) \oplus D_1(b) = 1$ and $D_1(ab)=1$.
		
		Similarly, there exists an isomorphism $D_2$ between $\Cli(\Sym[(l_n)_n])$ and $[0,1]^\bbN / \Fin$. Because of the definability showcased above, $\varphi$ now induces via $D_1$ and $D_2$ an automorphism of the structure $([0,1]^\bbN / \Fin,0,1,\leq,\neg,\oplus).$ However, all automorphisms of the latter structure are induced by autohomeomorphisms of $\partial \beta \bbN,$ via the identification $[0,1]^\bbN / \Fin = C(\partial \beta \bbN,[0,1])$ (see \cite{CIGNOLI200437}*{Proposition 4.2} and \cite{MR26240}). Hence, there exists an autohomeomorphism $\theta$ of the compact space $\partial \beta \bbN$ such that for all involutions $(a_n)_n \in \Sym[(k_n)_n]$, $(b_n)_n \in \Sym[(l_n)_n]$ and all $\cU \in \partial \beta \bbN:$
		\[ (b_n)_n = \varphi( (a_n)_n ) \quad \Rightarrow \quad \lim_{n \to \cU} d_{l_n}(1,b_n) = \lim_{n \to \theta(\cU)} d_{k_n}(1,a_n). \]
		
		By Stone duality, $\theta$ can be identified with an automorphism of the Boolean algebra $\cP(\bbN)/\Fin$ and the above statement after this identification 
		gives that: 
		\[
		a=_S 1 \Rightarrow \varphi(a)=_{\theta(S)} 1,
		\]
		for all involutions  $a \in \Sym[(k_n)_n]$ and $S \in \mathcal P(\bbN)/\Fin$. Now let $b$ be an arbitrary element of $\Sym[(k_n)_n]$, let $S \in \mathcal P(\bbN)/\Fin$ and assume that $b=_S 1$. Select an involution $a \in \Sym[(k_n)_n]$ such that the sequences $(d_{k_n}(1,a_n))_n$ and $(d_{k_n}(1,b_n))_n$ determine the same element of the metric reduced product $[0,1]^\bbN / \Fin$. It follows that $\varphi(a)=_{\theta(S)} 1$. Basic combinatorics shows that $b$ is a product of two conjugates of $a$.  Now, observe next that, since $\varphi$ is a group automorphism, $\varphi(b)$ can be written as the product of two conjugates of $\varphi(a)$, and therefore it follows that also $\varphi(b)=_{\theta(S)} 1$.
		
		This completes the proof of \eqref{itm1}. For property \eqref{itm2}, let $\cU$ be any nonprincipal ultrafilter, let $a = (a_n)_n,b =(b_n)_n \in \Sym[(k_n)_n]$ and  $\varphi(a) = (a'_n)_n, \varphi(b) = (b'_n)_n \in \Sym[(l_n)_n].$
		Then (still identifying $\theta$ both with an element of $ \Aut(\partial \beta \bbN)$ and with the associated element of $\Aut(\cP(\bbN) / \Fin)$), without making a distinction in notation):
		\[\
		\begin{aligned}
			\relax	[a]_{\theta(\cU)}=[b]_{\theta(\cU)}&\Leftrightarrow \{ n: a_n=b_n \}\in \theta(\cU)\\&\Leftrightarrow \theta \left(\{ n:a_n=b_n  \}\right)\in\cU\\
			&\Leftrightarrow \{ n: a'_n =b'_n\}\in\cU\\
			&\Leftrightarrow [\varphi(a)]_{\cU}=[\varphi(b)]_{\cU}.
		\end{aligned}
		\]
		Here, in the penultimate step we made use of \eqref{itm1}. It thus follows that $\varphi$ induces an isomorphism $\varphi_{\theta(\cU)}:  \prod_n  \Sym(k_n) / \theta(\cU) \to \prod_n  \Sym(l_n) / \cU$.
		
	\end{proof}

	By \cite{alekseevthom}*{Proposition 5.1}, for every ultrafilter $\cU \in \partial \beta \bbN$ the induced isomorphism  $\varphi_\cU: \prod_n  \Sym(k_n) / \cU \to \prod_n  \Sym(l_n) / \theta^{-1}(\cU)$ is isometric. Hence, we find the following corollary. 
	
	\begin{corollary}\label{cor.isometric}
		Every isomorphism $\varphi: \Sym[(k_n)_n] \to \Sym[(l_n)_n]$ is automatically isometric. Even more, for every infinite $S \in \cP(\bbN) / \Fin$ and every $a,b \in \Sym[(k_n)_n]:$ \[ d^S(a,b) = d^{\theta(S)}(\varphi(a),\varphi(b)).\] 
	\end{corollary}
	
	\begin{corollary}\label{cor.noniso}
		There exist precisely $2^{\aleph_0}$ many different metric reduced product groups $\Sym[(k_n)_n]$ up to isomorphism.
	\end{corollary}
	\begin{proof}
		By the proof of \cite{alekseevthom}*{Theorem 5.5}, one can construct a sequence of metric formulas $(\varphi_m)_{m\in \bbN}$ and for every $\nu \in \{0,1\}^\bbN$ a strictly increasing sequence $(k^\nu_n)_{n\in \bbN}$ of natural numbers such that for every $\nu  \in \{0,1\}^\bbN$ and every $m,$ the sequence of real numbers $\varphi_m^{\Sym(k_n^\nu)}$ converges to $\nu(m)$ as $n\to\infty$. It follows that when $\cU$ is a nonprincipal ultrafilter  on $\bbN$, the evaluation of the sentence $\varphi_m$ in the metric group $\Sym[(k^\nu_n)_n] / \cU$ is precisely $\nu(m).$ Therefore $\Sym[(k^\nu_n)_n] / \cU$ and $\Sym[(k^\xi_n)_n] / \mathcal{V}$ are never isomorphic for two different elements $\nu, \xi \in \{0,1\}^\bbN.$ It follows by Theorem~\ref{th.action boundary} that then the groups $\Sym[(k^\nu_n)_n]$ and $\Sym[(k^\xi_n)_n]$ are non-isomorphic as well.
	\end{proof}
	
	\begin{remark}
		Since there are much more ultrafilters on $\bbN$ than sequences of natural numbers, the number of isomorphism types of the ultraproduct sofic groups $\Sym(\cU)$ can, in contrast to the case of reduced product sofic groups, vary between models of set theory. Indeed, it was shown by Thomas in \cite{MR2607888} that assuming
		$\mathsf{\neg CH}$, there are $2^{2^{\aleph_0}}$-many isomorphism classes of metric ultraproducts of symmetric groups; see also \cite{MR2802082}*{Theorem~5.1} for a more general  $\mathsf{\neg CH}$-result.
	\end{remark}
	
	\subsection{Main result}
	
	We are now ready to put everything together and give a proof of our main result. Recall our definition of triviality of an isomorphism from Definition \ref{def.triv}.
	
	\begin{theorem} Assuming $\OCA$ and $\MAsl$, every isomorphism $\varphi:\Sym[(k_n)_n]\to\Sym[(l_n)_n]$ is trivial.
	\end{theorem}
	\begin{proof}
		Let $\varphi:\Sym[(k_n)_n]\to\Sym[(l_n)_n]$ be an isomorphism. 	By Theorem~\ref{th.action boundary}, there exists $\theta\in\Aut(\cP(\bbN)/\Fin)$ such that:
		$$\limsup_{n\in S} d_{k_n}(a_n,b_n) = 0 \quad \Leftrightarrow \quad \limsup_{n\in\theta(S)} d_{l_n}(a'_n,b'_n) = 0,$$ for all $a,b \in \Sym[(k_n)_n]$, $\varphi(a) = (a'_n)_n, \varphi(b) = (b'_n)_n$ and any infinite subset $S \subseteq \bbN$.
		By \cite{TrivIso}*{Theorem~1}, $\OCA$ implies that there is an almost permutation $f:\bbN\to\bbN$ such that $f^{-1}[S]=\theta(S)$ for every $S \in \cP(\bbN)/\Fin$. Recall from the introduction that $f^{-1}$ induces an isomorphism $\psi_{f^{-1}}: \Sym[(k_{f(n)})_n] \to \Sym[(k_n)_n]$. Consider next the isomorphism $\hat{\varphi}:= \varphi \circ \psi_{f^{-1}}:  \Sym[(k_{f(n)})_n] \to \Sym[(l_n)_n]$ and note that for all $a,b \in \Sym[(k_{f(n)})_n]$ and all infinite subsets $S \subseteq \bbN$, for $\hat \varphi(a) = (a_n')_n,\hat \varphi(b) = (b_n')_n$:
		$$\limsup_{n\in S} d_{k_{f(n)}}(a_n,b_n) = 0 \quad \Leftrightarrow \quad \limsup_{n\in S} d_{l_n}(a'_n,b'_n) = 0.$$
		The main theorem of \cite{metricliftingtheorem} implies that every such map $\hat{\varphi}$ is of product form. By the structure result for isomorphisms of product form, Lemma~\ref{lem.coordfixcase}, we have that $\lim \frac{k_{f(n)}}{l_n}=1$ and that there exists $\sigma \in \Sym[(l_n)_n]$ such that $\hat{\varphi}$ maps every element $(b_n)_n \in \Sym[(k_{f(n)})_n]$ to the element $\ad_\sigma((b_n \updownarrow \Sym(l_n))_n)$ of $\Sym[(l_n)_n].$
		It follows that $\varphi$ maps the element $(a_n)_n \in \Sym[(k_n)_n]$ precisely to the element $\ad_\sigma((a_{f(n)} \updownarrow \Sym(l_n))_n)$, hence $\varphi$ is trivial.
	\end{proof}
	
	\begin{definition}
		Let's call the sequence $(l_n)_n$ an \emph{asymptotic almost rearrangement} of $(k_n)_n$ if there exists an almost permutation $f: \bbN \to \bbN$ such that $\displaystyle\lim_{n\to\infty}\nicefrac{l_n}{k_{f(n)}}=1$. Call the sequence $(l_n)_n$ an \emph{asymptotically equivalent almost rearrangement} of $(k_n)_n$ if $\displaystyle\lim_{n\to\infty}\nicefrac{l_n}{k_n}=1$ and there exists an almost permutation $f: \bbN \to \bbN$ such that $(k_{f(n)})_n = (l_n)_n.$ 
	\end{definition}
	
	We obtain the following immediate corollary.
	\begin{corollary} Assume $\OCA$ and $\MAsl$.
		There exists an isomorphism $\Sym[(k_n)_n] \cong \Sym[(l_n)_n]$ if and only if $(k_n)_n$ and $(l_n)_n$ are asymptotic almost rearrangements of each other.    
	\end{corollary}
	
	\begin{definition}
		We denote by ${\rm Out}((k_n)_n)$ the group of almost permutations $f \colon \bbN \to \bbN$ modulo equality on co-finite sets such that $$\displaystyle\lim_{n\to\infty}\frac{k_{f(n)}}{k_n}=1.$$
	\end{definition}
	
	The following corollary gives a complete description of the automorphism group of a metric reduced product of symmetric groups.
	
	\begin{corollary} Assume $\OCA$ and $\MAsl$.
		\begin{enumerate}
			\item If the sequence $(k_n)_n$ does not have any nontrivial asymptotically equivalent almost rearrangements, then $\Sym[(k_n)_n]$ is a complete group, i.e., all automorphisms are inner. This happens if and only if $$\liminf_{n \neq m} |\log(k_{n}) - \log(k_{m})|>0.$$
			\item The group of outer automorphisms of $\Sym[(k_n)_n]$ is isomorphic to ${\rm Out}((k_n)_n)$ and the group extension
			$$1 \to \Sym[(k_n)_n] \to {\rm Aut}(\Sym[(k_n)_n]) \to {\rm Out}((k_n)_n) \to 1$$ is split.
		\end{enumerate}
	\end{corollary}	
	
	One particular consequence of these main results is that under forcing axioms there exist (many) elementarily equivalent but non-isomorphic groups $\Sym[(k_n)_n]$ and $\Sym[(l_n)_n]$.
	
	\begin{proposition} \label{Prop.eleqnotiso}
		Assuming $\OCA$ and $\MAsl$, there exist sequences $(k_n)_n$ and $(l_n)_n$ with $\lim_n k_n = \lim_n l_n = \infty$ such that the groups $\Sym[(k_n)_n]$ and $\Sym[(l_n)_n]$ are elementarily equivalent (in the sense of metric model theory) but not isomorphic.
	\end{proposition}
	\begin{proof}
		We follow a standard compactness argument (see \cite{TrivIso}*{proof of Proposition 6} for this argument in the discrete case which directly generalizes to the continuous setting).
		First note that if the sequences of metric theories $(\Th(\Sym(k_n)))_n$ and $(\Th(\Sym(l_n)))_n$ converge to the same limit in the logic topology, then by the metric Feferman-Vaught theorem (\cite{ghasemi2014reduced}) the metric groups $\Sym[(k_n)_n]$ and $\Sym[(l_n)_n]$ are elementarily equivalent.
		It follows that for every sequence $(m_n)_n$ of natural numbers with ${\lim_n m_n=\infty},$ there exists an infinite $S \subseteq \bbN$ such that for all infinite subsets $T_1,T_2 \subseteq S$:
		\[ \Sym[(m_n)_{n \in T_1}] \equiv \Sym[(m_n)_{n \in T_2}].\] Indeed, the space of theories in the (separable) language of metric groups is compact and metrizable, hence there exists an infinite $S\subseteq \bbN$  such that $(\Th(\Sym(m_n)))_{n\in S}$ converges and then by consequence we have for all infinite subsets $T_1,T_2 \subseteq S$ that
		\[ \Sym[(m_n)_{n \in T_1}] \equiv \Sym[(m_n)_{n \in T_2}].\]
		Any two sequences $(k_n)_n$ and $(l_n)_n$ that enumerate infinite subsets of $\{ m_n : n \in S \}$ and satisfy $\liminf_{n\neq n'} |\log(k_n) - \log(l_{n'})| >0$ now have the required properties.
	\end{proof}
	
	One can compare our main results to the following standard $\mathsf{CH}$-results, the proofs of which all follow well-known arguments.
	
	\begin{theorem}\label{Th.CHcase}
		Assume $\mathsf{CH}$.
		
		\begin{enumerate}
			\item\label{itm-1}     Every group $ \Sym[(k_n)_n]$ has $2^{2 ^ {\aleph_0}} = 2^{\aleph_1}$ many (isometric) automorphisms that are not trivial,
			\item\label{itm-2} the groups $ \Sym[(k_n)_n]$ and $ \Sym[(l_n)_n]$ are (isometrically) isomorphic if and only if they are elementarily equivalent (in the sense of continuous model theory),
			\item \label{itm-3} for every sequence $(k_n)_n$ of natural numbers with $\lim_n k_n = \infty,$ there exists an infinite $S \subseteq \bbN$ such that for all infinite subsets $T_1,T_2 \subseteq S$:
			\[ \Sym[(k_n)_{n \in T_1}] \cong \Sym[(k_n)_{n \in T_2}],\]
			\item\label{itm-4} there exist sequences $(k_n)_n$ and $(l_n)_n$ with $\lim_n k_n = \lim_n l_n=\infty,$ such that $\Sym[(k_n)_{n \in T_1}] \cong \Sym[(k_n)_{n \in T_2}],$ yet $(l_n)_n$ is no asymptotic almost rearrangement of $(k_n)_n$.
		\end{enumerate}
	\end{theorem}

	\begin{proof}
		Both \eqref{itm-1} and \eqref{itm-2} follow straight from the fact that the reduced product groups $\Sym[(k_n)_n]$ are all countably saturated (\cite{Fa:Combinatorial}*{Theorem 16.5.1}) and under $\mathsf{CH}$ have density character $\aleph_1$  (see resp.\ \cite{Fa:Combinatorial}*{Theorem 16.6.3} and \cite{Fa:Combinatorial}*{Theorem 16.6.5}). Using \eqref{itm-2}, the items \eqref{itm-3} and \eqref{itm-4} now easily follow from the construction in the proof of Proposition~\ref{Prop.eleqnotiso}.
	\end{proof}
	
	\section{Open problems and remarks}
	
	\begin{question}
		Is there a nice characterisation, in $\mathsf{ZFC}$, of those autohomeomorphisms $\theta$ of $\partial \beta \bbN$ which are induced by some automorphism of $\Sym[(k_n)_n]$?
	\end{question}
	
	\begin{question}
		Can one classify all homomorphisms $\varphi: \Sym[(k_n)_n] \to \Sym[(l_n)_n]$ under suitable set-theoretic assumptions?
	\end{question}
	
	\begin{question}
		For which sequences $(k_n)_n$, $(l_n)_n$ are  $\Sym[(k_n)_n]$ and $\Sym[(l_n)_n]$ elementarily equivalent?    
	\end{question}
	
	\begin{question}
		Is it consistent with $\mathsf{ZFC}$ that there exists an ultrafilter $\cU$ on $\bbN$ such that all automorphisms of the metric ultraproduct $\prod_\cU \Sym(n)$ are inner?  
	\end{question}
	
	\begin{remark}
		It seems natural to expect that similar results hold for the groups ${\rm U}(n)$ with the normalized trace-norm and with the operator norm in place of the symmetric group. Note that the corresponding reduced product in the norm case is the unitary group of the well-studied Corona algebra, see \cites{Fa:Combinatorial, farah2024coronarigidity}. The required Ulam stability results were proved for the normalized trace-norm by Gowers--Hatami \cite{MR3733361} and for the operator norm by Grove--Karcher--Ruh \cite{MR367862}, see also \cite{MR3867328}. This is subject of ongoing work.
	\end{remark}
	
	\section*{Acknowledgments}
	
	Ben De Bondt is funded by the Deutsche Forschungsgemeinschaft (DFG, German Research Foundation) under Germany's Excellence Strategy EXC 2044\,-390685587, Mathematics Münster: Dynamics--Geometry--Structure. 
	
	Andreas Thom acknowledges funding by the Deutsche Forschungsgemeinschaft (SPP 2026 ``Geometry at infinity'').
	
	We both thank Ilijas Farah for inspiring discussions about these topics and many useful remarks.

\end{document}